\documentclass[11pt]{amsart}
\usepackage[english]{babel}
\usepackage{amssymb}
\usepackage{amsmath}
\usepackage{amsthm}
\usepackage{bbold}
\usepackage{bbm}
\usepackage{mathtools}
\usepackage{dsfont}
\usepackage{color}
\usepackage[scr=rsfso]{mathalfa}
\newtheorem{thm}{Theorem}

\newtheorem{exam}{Example}
\newtheorem{cor}{Corollary}

\newtheorem{defi}{Definition}
\newtheorem{ques}{Question}
\numberwithin{equation}{section}
\DeclareSymbolFont{bbold}{U}{bbold}{m}{n}

\DeclareMathOperator{\converg}{\xrightarrow[]{c}}
\DeclareMathOperator{\AS}{\xrightarrow[]{a.e.}}
\DeclareMathOperator{\convergX}{\xrightarrow[]{c_X}}
\DeclareMathOperator{\convergY}{\xrightarrow[]{c_Y}}
\DeclareMathOperator{\CX}{\xrightarrow[]{c_X}}
\DeclareMathOperator{\CE}{\xrightarrow[]{c_E}}

\DeclareMathOperator{\OF}{\xrightarrow[]{o_F}}
\DeclareMathOperator{\uoc}{\xrightarrow[]{uo}}
\DeclareMathOperator{\oc}{\xrightarrow[]{o}}

\DeclareMathOperator{\mc}{\xrightarrow[]{\mathcal{M}}}

\renewcommand{\le}{\leqslant}
\renewcommand{\ge}{\geqslant}

\begin{document}

\title{AMS Journal Sample}
	\author{E. Y. Emelyanov$^{1,2}$}
	\address{$^1$ Middle East Technical University, 06800 Ankara, Turkey}
	\email{eduard@metu.edu.tr}
	\address{$2$ Sobolev Institute of Mathematics, 630090 Novosibirsk, Russia}
	\email{emelanov@math.nsc.ru}
	\author{M. Marabeh$^3$}
	\address{$^3$Department of Applied Mathematics, College of Sciences and Arts, Palestine Technical University-Kadoorie, Tulkarem, Palestine}
	\email{mohammad.marabeh@ptuk.edu.ps, m.maraabeh@gmail.com}
	
	\date{\today}
	
\title{Brezis - Lieb spaces and an operator version of Brezis - Lieb's lemma}

\begin{abstract}
The Brezis - Lieb spaces, in which Brezis - Lieb's lemma holds true for 
nets, are introduced and studied. An operator version of Brezis - Lieb's  
lemma is also investigated. 
\end{abstract}

\maketitle
\date{\today}

\section{Introduction}

Throughout the paper, $(\Omega,\Sigma,\mu)$ stands for a measure space in which every set $A\in\Sigma$ of nonzero measure possesses
a subset $A_0\subseteq A$, $A_0\in\Sigma$, such that $0<\mu(A_0)<\infty$. The famous Brezis - Lieb lemma \cite[Thm.2]{BL1} 
is known as Theorem \ref{Thm1} \cite[Thm.2]{BL1}, and as its corollary, Theorem \ref{Thm2} \cite[Thm.1]{BL1}, 
and also as Theorem \ref{Thm3} (cf. \cite[Cor.3]{N1}), which is a corollary of Theorem \ref{Thm2}. 

\begin{thm}[Brezis - Lieb's lemma]\label{Thm1}
Let $j:\mathbb{C}\to\mathbb{C}$ be a continuous function with $j(0)=0$ such that, for every $\varepsilon>0$, there exist 
two non-negative continuous functions $\phi_\varepsilon, \psi_\varepsilon: \mathbb{C}\to\mathbb{R}_+$ with
\begin{equation}\label{1.1}
  |j(x+y)-j(x)|\le\varepsilon\phi_{\varepsilon}(x)+\psi_{\varepsilon}(y) \ \ \ \ \ (\forall x,y\in \mathbb{C}).
\end{equation}
Let $g_n$ and $f$ be $($$\mathbb{C}-$valued$)$ functions in $\mathcal{L}^0(\mu)$ such that $g_n\AS 0$; $j(f)$, $\phi_{\varepsilon}(g_n)$, 
$\psi_{\varepsilon}(f)\in \mathcal{L}^1(\mu)$ for all $\varepsilon>0$, $n\in\mathbb{N}$; and let 
$$
  \sup\limits_{\varepsilon>0,n\in\mathbb{N}} \ \int\limits_\Omega\phi_{\varepsilon}(g_n(\omega))d\mu(\omega)\le C<\infty.
$$
Then
\begin{equation}\label{1.2}
  \lim\limits_{n\to\infty}\int\limits_\Omega|j(f+g_n)-(j(f)+j(g_n))|d\mu(\omega)=0.
\end{equation}
\end{thm}

For a proof of Theorem \ref{Thm1}, see \cite{BL1}.

\begin{thm}[Brezis - Lieb's lemma for $\mathcal{L}^p$ $(0<p<\infty)$]\label{Thm2}
Suppose $f_n\AS f$ and $\int\limits_{\Omega}|f_n|^pd\mu\le C<\infty$ for all $n$ and some $p\in(0,\infty)$. Then 
\begin{equation}\label{1.3}
  \lim\limits_{n\to\infty}\bigg\{\int\limits_{\Omega}\bigg(|f_n|^p-|f_n-f|^p\bigg)d\mu\bigg\}=\int\limits_{\Omega}|f|^pd\mu. 
\end{equation}
\end{thm}

We reproduce here the arguments from \cite{BL1} since they are short and instructive. Take $j(z)=\phi_{\varepsilon}(z):=|z|^p$ 
and $\psi_{\varepsilon}(z)=C_{\varepsilon}|z|^p$ for a sufficiently large $C_{\varepsilon}$. Theorem \ref{Thm1} applied 
to $g_n=f_n-f$ ensures $f\in \mathcal{L}^p(\mu)$, which, in view of (\ref{1.2}), completes the proof of Theorem \ref{Thm2}.
Theorem \ref{Thm3} below is an immediate corollary of Theorem \ref{Thm2} (cf. also \cite[Cor.3]{N1}).

\begin{thm}[Brezis - Lieb's lemma for $L^p$ $(1\le p<\infty)$]\label{Thm3}
Let $\mathbf{f}_n\AS \mathbf{f}$ in $L^p(\mu)$ and $\|\mathbf{f}_n\|_p\to\|\mathbf{f}\|_p$, where 
$\|\mathbf{f}_n\|_p:=\bigg[{\int\limits_{\Omega}|f_n|^pd\mu}\bigg]^{1/p}$ with 
$f_n \in \mathcal{L}^p(\mu)$ and $ f_n\in\mathbf{f}_n$. Then $\|\mathbf{f}_n-\mathbf{f}\|_p\to 0$.
\end{thm}

Theorem \ref{Thm3} is a Banach lattice type result if $a.e.-$convergence is replaced 
by $uo-$convergence (cf. \cite[Prop.3.1]{GTX1}). It motivates us to investigate the 
general class of Banach lattices, in which the statement of Theorem \ref{Thm3} yields.
Even more important reason for such investigation relies on the fact that all the above 
versions of Brezis - Lieb's lemma in Theorems \ref{Thm1}, \ref{Thm2}, and \ref{Thm3}, 
are sequential due to the sequential nature of $a.e.-$convergence. 
It is worth to mention that Corollary \ref{BL for nets in L^p} may serves 
as an extension of the Brezis - Lieb lemma (in form of Theorem \ref{Thm3}) for nets.  
 
In Section 2, we introduce Brezis - Lieb's spaces and their sequential version. Then we
prove Theorem \ref{BLS} which gives an internal geometric characterization of Brezis - Lieb's 
spaces. We also discuss possible extensions of Theorem \ref{BLS} to locally solid Riesz spaces. 

In Section 3, we prove Theorem \ref{OVBL-NV} which can be seen as an operator version 
of Theorem \ref{Thm1} in convergence spaces.

For the theory of vector lattices we refer to \cite{AB1,AB2} and for unbounded convergences to \cite{DE1,DEM1,DEM2,GX1,GTX1,GLX1}.

\section{Brezis - Lieb spaces}

\begin{defi}
A normed lattice $(E,\|\cdot\|)$ is said to be a {\em Brezis - Lieb space $($shortly, a $BL-$space$)$} 
$($resp. {\em $\sigma$-Brezis - Lieb space $($$\sigma$-$BL-$space$)$}$)$ if, for any net $x_\alpha$ 
$($resp, for any sequence $x_n$$)$ in $X$ such that $\|x_\alpha\|\to\|x_0\|$ $($resp. $\|x_n\|\to\|x_0\|$$)$
and $x_\alpha\uoc x_0$ $($resp. $x_n\uoc x_0$$)$ we have that $\|x_\alpha-x_0\|\to 0$ $($resp. $\|x_n-x_0\|\to 0$$)$.
\end{defi}

Trivially, any normed Brezis - Lieb space is a $\sigma$-BL-space, and any finite-dimensional normed lattice is a $BL$-space.
Taking into account that $a.e.-$con\-ver\-gence for sequences in $L^p$ is the same as $uo-$conver\-gence \cite[Prop.3.1]{GTX1}, 
Theorem \ref{Thm3} says exactly that $L^p$ is a $\sigma$-BL-space for $1\le p<\infty$. 

\begin{exam}\label{c_0 is not sigma-BL-space}
The Banach lattice $c_0$ is not a $\sigma-$Brezis - Lieb space. To see this, take 
$x_n=e_{2n}+\sum\limits_{k=1}^{n}\frac{1}{k}e_k$ and $x=\sum\limits_{k=1}^{\infty}\frac{1}{k}e_k$ in $c_0$.
Clearly, $\|x\|=\|x_n\|=1$ for all $n$ and $x_n\uoc x$, however $1=\|x-x_n\|$ does not converge to $0$.   
\end{exam}

A slight change of an infinite-dimensional BL-space may turn it 
into a normed lattice which is even not a $\sigma-$BL-space. 

\begin{exam}\label{BL-spaces are not stable}
Let $E$ be a Brezis - Lieb space, $\dim(E)=\infty$. Let $E_1=\mathbb{R}\oplus_{\infty}E$.
Take any disjoint sequence $(y_n)_{n=1}^{\infty}$ in $E$ such that $\|y_n\|_E \equiv 1$. Then $y_n\uoc 0$ in $E$ \cite[Cor.3.6]{GTX1}.
Let $x_n=(1,y_n)\in E_1$. Then $\|x_n\|_{E_1}=\sup(1,\|y_n\|_E)=1$ and $x_n=(1,y_n)\uoc (1,0)=:x$ in $E_1$, however 
$\|x_n-x\|_{E_1}=\|(0,y_n)\|_{E_1}=\|y_n\|_E=1$ and so, $x_n$ does not converge to $x$ in $(E_1,\|\cdot\|_{E_1})$.
Therefore $E_1=\mathbb{R}\oplus_{\infty}E$ is not a $\sigma-$Brezis - Lieb space.
\end{exam}

In order to characterize BL-spaces, we introduce the following definition.

\begin{defi}\label{BL-property}
A normed lattice $(E,\|\cdot\|)$ is said to have the {\em Brezis - Lieb property $($shortly, $BL$-property$)$}, whenever
$\limsup\limits_{n\to\infty}\|u_0+u_n\|>\|u_0\|$ for any disjoint normalized 
sequence $(u_n)_{n=1}^{\infty}$ in $E_+$ and for any $u_0\in E$, $u_0>0$.
\end{defi}

Clearly, every finite dimensional normed lattice $E$ has the $BL-$property.
The Banach lattice $c_0$ obviously does not have the $BL-$property. The modification 
of the norm in an infinite-dimensional Banach lattice $E$ with the $BL-$property,
as in Example \ref{BL-spaces are not stable}, turns it into a Banach lattice $E_1=\mathbb{R}\oplus_{\infty}E$
without the $BL-$property. Indeed, take a disjoint normalized sequence $(y_n)_{n=1}^{\infty}$ in $E_+$.
Let $u_0=(1,0)$ and $u_n=(0,y_n)$ for $n\ge 1$. Then $(u_n)_{n=0}^{\infty}$ is a disjoint normalized sequence
in $(E_{1})_+$ with $\limsup\limits_{n\to\infty}\|u_0+u_n\|=1$.
Remarkably, it is not a coincidence.

\begin{thm}\label{BLS}
For a $\sigma-$Dedekind complete Banach lattice $E$, the following conditions are equivalent$:$\\
$(1)$ $E$ is a Brezis - Lieb space$;$\\
$(2)$ $E$ is a $\sigma-$Brezis - Lieb space$;$\\
$(3)$ $E$ has the $BL$-property and the norm in $E$ is order continuous.
\end{thm}

\begin{proof}
$(1)\Rightarrow (2)$ It is trivial.

$(2)\Rightarrow (3)$ We show first that $E$ has $BL$-property. Notice that in this part of the proof  
the $\sigma-$Dedekind completeness of $E$ will not be used. Suppose that there exists a disjoint normalized 
sequence $(u_n)_{n=1}^{\infty}$ in $E_+$ and a $u_0>0$ in $E$ with $\limsup\limits_{n\to\infty}\|u_0+u_n\|=\|u_0\|$.
Since $\|u_0\|\le\|u_0+u_n\|$, then $\lim\limits_{n\to\infty}\|u_0+u_n\|=\|u_0\|$.
Denote $v_n:=u_0+u_n$. By \cite[Cor.3.6]{GTX1}, $u_n\uoc 0$ and hence $v_n\uoc u_0$. 
Since $E$ is a $\sigma$-BL-space and $\lim\limits_{n\to\infty}\|v_n\|=\|u_0\|$, 
then $\|v_n-u_0\|\to 0$, which is impossible in view of $\|v_n-u_0\|=\|u_0+u_n-u_0\|=\|u_n\|=1$.

Assume that the norm in $E$ is not order continuous. Then, by the Fremlin--Meyer-Nieberg theorem 
(see for example \cite[Thm.4.14]{AB2}) there exist $y\in E_+$ and a disjoint sequence $e_k\in[0,y]$ 
such that $\|e_k\|\not\to 0$. Without lost of generality, we may assume $\|e_k\|=1$ for all $k\in\mathbb{N}$. 
By the $\sigma-$Dedekind completeness of $E$, for any sequence $\alpha_n\in\mathbb{R}_+$ there exist the following
vectors 
\begin{equation}\label{2.1}
  x_0=\bigvee\limits_{k=1}^{\infty}e_k, \ \ \  x_n=\alpha_{2n}e_{2n}+\bigvee\limits_{k=1,k\ne n,k\ne 2n}^{\infty}e_k 
	\ \ \ \ \ \ (\forall n\in \mathbb{N}).
\end{equation}
Now, we choose $\alpha_{2n}\ge 1$ in (\ref{2.1}) such that $\|x_n\|=\|x_0\|$ for all $n\in \mathbb{N}$.
Clearly, $x_n\uoc x_0$. Since $E$ is a $\sigma$-BL-space then $\|x_n-x_0\|\to 0$, violating 
$$
  \|x_n-x_0\|=\|(\alpha_{2n}-1)e_{2n}-e_n\|=\|(\alpha_{2n}-1)e_{2n}+e_n\|\ge\|e_n\|=1. 
$$
Obtained contradiction shows that the norm in $E$ is order continuous.

$(3)\Rightarrow (1)$  
If $E$ is not a Brezis - Lieb space, then there exists a net $(x_\alpha)_{\alpha\in A}$ in $E$ 
such that $x_\alpha\uoc x$ and $\|x_\alpha\|\to\|x\|$ but 
$\|x_\alpha-x\|\not\to 0$. Then $|x_\alpha|\uoc |x|$ and $\||x_\alpha|\|\to\||x|\|$. 

Notice that $\||x_\alpha|-|x|\|\not\to 0$. Indeed, if $\||x_\alpha|-|x|\|\to 0$ 
then $(x_\alpha)_{\alpha\in A}$ is eventually in $[-|x|,|x|]$ and then $(x_\alpha)_{\alpha\in A}$
is almost order bounded. Since $E$ is order continuous and $x_\alpha\uoc x$, 
then by \cite[Pop.3.7.]{GX1} $\|x_\alpha-x\|\to 0$, which is impossible.
Therefore, without lost of generality, we may assume that $x_\alpha\in E_+$ and, 
by normalizing, also $\|x_\alpha\|=\|x\|=1$ for all $\alpha$.

Passing to a subnet, denoted again by $x_\alpha$, we may assume
\begin{equation}\label{2.2}
  \|x_\alpha-x\|>C>0 \ \ \ (\forall \alpha\in A).
\end{equation}
Notice that $x\ge(x-x_\alpha)^+=(x_\alpha-x)^-\uoc 0$, and hence $(x_\alpha-x)^-\oc 0$. 
The order continuity of the norm ensures 
\begin{equation}\label{2.3}
  \|(x_\alpha-x)^-\|\to 0.
\end{equation}
Denoting $w_\alpha=(x_\alpha-x)^+$ and using (\ref{2.2}) and (\ref{2.3}), we may also assume 
\begin{equation}\label{2.4}
  \|w_\alpha\|=\|(x_\alpha-x)^+\|>C \ \ \ (\forall \alpha\in A).
\end{equation}
In view of (\ref{2.4}), we obtain
\begin{equation}\label{2.5}
  2=\|x_\alpha\|+\|x\|\ge\|(x_\alpha-x)^+\|=\|w_\alpha\|>C \ \ \ (\forall \alpha\in A).
\end{equation}
Since $w_\alpha\uoc(x-x)^+=0$ then, for any fixed $\beta_1,\beta_2,...,\beta_n$,
\begin{equation}\label{2.6}
  0\le w_{\alpha}\wedge(w_{\beta_1}+w_{\beta_2}+...+w_{\beta_n})\oc 0 \ \ \ \ (\alpha\to\infty).
\end{equation}
Since $x_\alpha\uoc x$, then $x_\alpha\wedge x\uoc x\wedge x=x$ and so $x_\alpha\wedge x\oc x$. By the order continuity of the norm,
there is an increasing sequence of indices $\alpha_n$ in $A$ with
\begin{equation}\label{2.7}
  \|x-x_\alpha\wedge x\|\le 2^{-n} \ \ \ \ (\forall \alpha\ge\alpha_n).
\end{equation}
Furthermore, by (\ref{2.6}), we may also suppose that
\begin{equation}\label{2.8}
  \|w_{\alpha}\wedge(w_{\alpha_1}+w_{\alpha_2}+...+w_{\alpha_n})\|\le 2^{-n}\ \ \ \ (\forall\alpha\ge\alpha_{n+1}). 
\end{equation}
Since 
$$
  \sum\limits_{k=1,k\ne n}^{\infty}\|w_{\alpha_n}\wedge w_{\alpha_k}\|\le
	\sum\limits_{k=1}^{n-1}\|w_{\alpha_n}\wedge (w_{\alpha_1}+...+w_{\alpha_{n-1}})\|+
$$
$$	
	\sum\limits_{k=n+1}^{\infty}\|w_{\alpha_k}\wedge (w_{\alpha_1}+...+w_{\alpha_{k-1}})\|\le
	(n-1)\cdot 2^{-n+1}+\sum\limits_{k=n+1}^{\infty}2^{-k+1}=n2^{-n+1},
	\eqno(2.9)
$$
the series $\sum\limits_{k=1,k\ne n}^{\infty}w_{\alpha_n}\wedge w_{\alpha_k}$ converges absolutely and hence in norm for any $n\in\mathbb{N}$. Take
$$
  \omega_{\alpha_n}:=\bigg(w_{\alpha_n}-\sum\limits_{k=1,k\ne n}^{\infty}w_{\alpha_n}\wedge w_{\alpha_k}\bigg)^+ 
	\ \ \ \ (\forall n\in\mathbb{N}).
	\eqno(2.10)
$$
First, we show that the sequence $(\omega_{\alpha_n})_{n=1}^{\infty}$ is disjoint. Let $m\ne p$, then
$$
  \omega_{\alpha_m}\wedge\omega_{\alpha_p}=
	\bigg(w_{\alpha_m}-\sum\limits_{k=1,k\ne m}^{\infty}w_{\alpha_m}\wedge w_{\alpha_k}\bigg)^+\wedge
	\bigg(w_{\alpha_p}-\sum\limits_{k=1,k\ne p}^{\infty}w_{\alpha_p}\wedge w_{\alpha_k}\bigg)^+\le
$$
$$
	(w_{\alpha_m}-w_{\alpha_m}\wedge w_{\alpha_p})^+\wedge(w_{\alpha_p}-w_{\alpha_p}\wedge w_{\alpha_m})^+=
$$
$$	
	(w_{\alpha_m}-w_{\alpha_m}\wedge w_{\alpha_p})\wedge(w_{\alpha_p}-w_{\alpha_m}\wedge w_{\alpha_p})=0.
$$
By (2.9),
$$
  \|w_{\alpha_n}-\omega_{\alpha_n}\|=\bigg\|w_{\alpha_n}-\bigg(w_{\alpha_n}-\sum\limits_{k=1,k\ne n}^{\infty}w_{\alpha_n}\wedge w_{\alpha_k}\bigg)^+\bigg\|=
$$
$$
  \bigg\|w_{\alpha_n}-\bigg(w_{\alpha_n}-w_{\alpha_n}\wedge\sum\limits_{k=1,k\ne n}^{\infty}w_{\alpha_n}\wedge w_{\alpha_k}\bigg)\bigg\|=
  \bigg\|w_{\alpha_n}\wedge\sum\limits_{k=1,k\ne n}^{\infty}w_{\alpha_n}\wedge w_{\alpha_k}\bigg\|\le
$$
$$
  \|\sum\limits_{k=1,k\ne n}^{\infty}w_{\alpha_n}\wedge w_{\alpha_k}\|\le n2^{-n+1}.\ \ \ \ (\forall n\in\mathbb{N}). 
	\eqno(2.11)
$$
Combining (2.11) with (\ref{2.5}) gives 
$$
  2\ge\|w_{\alpha_n}\|\ge\|\omega_{\alpha_n}\|\ge C-n2^{-n+1}\ \ \ \ (\forall n\in\mathbb{N}). 
  \eqno(2.12)
$$
Passing to further increasing sequence of indices, we may assume that
$$
  \|w_{\alpha_n}\|\to M\in[C,2] \ \ \ \ (n\to\infty).
$$
Now
$$
  \lim\limits_{n\to\infty}\bigg\|M^{-1}x+\|\omega_{\alpha_n}\|^{-1}\omega_{\alpha_n}\bigg\|=
	M^{-1}\lim\limits_{n\to\infty}\|x+\omega_{\alpha_n}\|=[\text{by} \ (2.11)]=
$$
$$	
	M^{-1}\lim\limits_{n\to\infty}\|x+w_{\alpha_n}\|=[\text{by} \ (2.3)]=
	M^{-1}\lim\limits_{n\to\infty}\|x+(x_{\alpha_n}-x)\|=
$$
$$
	M^{-1}\lim\limits_{n\to\infty}\|x_{\alpha_n}\|=M^{-1}=\|M^{-1}x\|,
$$
violating the the Brezis - Lieb property for $u_0=M^{-1}x$ and $u_n=\|\omega_{\alpha_n}\|^{-1}\omega_{\alpha_n}$,
$n\ge 1$. The obtained contradiction completes the proof.
\end{proof}

The next fact is a corollary of Theorem \ref{BLS} which states a {\em Brezis - Lieb's type lemma for nets in $L^p$}. 

\begin{cor}\label{BL for nets in L^p}
Let $\mathbf{f}_\alpha\uoc\mathbf{f}$ in $L^p(\mu)$, $(1\le p<\infty)$, and $\|\mathbf{f}_\alpha\|_p\to\|\mathbf{f}\|_p$.
Then $\|\mathbf{f}_\alpha-\mathbf{f}\|_p\to 0$.
\end{cor}

We do not know where or not implication $(2)\Rightarrow (3)$ of Theorem \ref{BLS} holds true without the assumption
that the Banach lattice $E$ is $\sigma-$Dedekind complete. 

\begin{ques}
Does every $\sigma-$Brezis - Lieb Banach lattice have order continuous norm?
\end{ques}

In the proof of $(2)\Rightarrow (3)$ $\sigma-$Dedekind completeness of $E$ has been used only 
for showing that $E$ has order continuous norm. So, any $\sigma-$Brezis - Lieb Banach lattice has
the Brezis - Lieb property. Therefore, for answering in positive the question of possibility 
to drop $\sigma-$Dedekind completeness assumption in Theorem \ref{BLS}, it is sufficient 
to have the positive answer to the following question. 

\begin{ques}
Does the Brezis - Lieb property imply order continuity of the norm?
\end{ques}

In the end of the section we discuss possible generalizations of Brezis - Lieb spaces and Brezis - Lieb property. 
To avoid overloading the text, we restrict ourselves with the case of multi-normed Brezis - Lieb lattices, 
postponing the discussion of locally solid Brezis - Lieb lattices to further papers.

A multi-normed vector lattice (shortly, MNVL) $E=(E,\mathcal{M})$ (see \cite{DEM1}):\\
$(a)$ is said to be a {\em Brezis - Lieb space} if
$$
  [x_\alpha\uoc x_0 \ \ \& \ \ m(x_\alpha)\to m(x_0) \ \ (\forall m\in\mathcal{M})]\Rightarrow[x_\alpha\mc x_0].  
$$
$(b)$ has the {\em Brezis - Lieb property}, if for any disjoint sequence 
$(u_n)_{n=1}^{\infty}$ in $E_+$ such $u_n$ does not converge in $\mathcal{M}$ to 0 and for any $u_0>0$,
there exists $m\in\mathcal{M}$ such that $\limsup\limits_{n\to\infty}m(u_0+u_n)>m(u_0).$

A $\sigma$-Brezis - Lieb MNVL is defined by replacing of nets with sequences.

By using the above definitions one can derive from Theorem \ref{BLS} the following result, whose details are left to the reader. 

\begin{cor}\label{BLMNL}
For an MNVL $E$ with a separating order continuous multinorm $\mathcal{M}$, the following conditions are equivalent$:$\\
$(1)$ $E$ is a Brezis - Lieb space$;$\\
$(2)$ $E$ is a $\sigma-$Brezis - Lieb space$;$\\
$(3)$ $E$ has the Brezis - Lieb property.
\end{cor}

\section{Operator version of Brezis - Lieb's lemma in convergent vector spaces}

In this section, we consider both complex and real vector spaces and vector lattices.   
A {\em convergence} ``$\converg$" for nets in a set $X$ is defined by the following conditions:\\
$(a)$ \ $x_\alpha\equiv x\Rightarrow x_\alpha\converg x$, and\\ 
$(b)$ \ $x_\alpha\converg x\Rightarrow x_\beta\converg x$ for every subnet $x_\beta$ of $x_\alpha$.\\
A mapping $f$ from a {\em convergence set} $(X,c_X)$ into a convergence set $(Y,c_Y)$ is said to be $c_Xc_Y-${\em continuous} 
(or just continuous), if $x_\alpha\convergX x$ implies $f(x_\alpha)\convergY f(x)$ for every net $x_\alpha$ in $X$. 

A subset $A$ of $(X,c_X)$ is called {\em $c_X-$closed} if $A\ni x_\alpha\convergX x\Rightarrow x\in A$.
If the set $\{x\}$ is $c_X-$closed for every $x\in X$ then $c_X$ is called {\em $T_1$-convergence}.
It is immediate to see that $c_X\in T_1$ iff every constant net $x_\alpha\equiv x$ does not $c_X-$converge to any $y\ne x$.

Under {\em convergence vector space} $(X,c_X)$ we understand a vector space $X$ with a convergence $c_X$ such that
the linear operations in $X$ are $c_X-$continuous. $(E,c_E)$ is a {\em convergence vector lattice} if $(E,c_E)$ is a convergence 
vector space which is a vector lattice where the lattice operations are also $c_E-$continuous. For further references see \cite{AB1,AB2,DE1}.

Motivated by the proof of the famous Brezis - Lieb's lemma \cite[Thm.2]{BL1}, we present its operator version in convergent spaces.

Given a convergence complex vector space $(X,c_X)$; two convergence complex vector lattices $(E,c_E)$ and $(F,c_F)$, where $F$ is Dedekind complete;
an order ideal $E_0$ in $E_+-E_+$; and a $c_{E_0}o_F-$continuous positive linear operator $T:E_0\to F$, where $o_F$ stands for the order 
convergence in $F$. Furthermore, let $J:X\to E$ be $c_Xc_E-$continuous, $J(0)=0$, and, for every $\varepsilon>0$, let there exist 
two $c_Xc_E-$continuous mappings $\Phi_\varepsilon,\Psi_\varepsilon:X\to E_+$ with
\begin{equation}\label{3.1}
  |J(x+y)-Jx|\le\varepsilon\Phi_{\varepsilon}x+\Psi_{\varepsilon}y \ \ \ \ \ (\forall x,y\in X).
\end{equation}

\begin{thm}[An operator version of Brezis - Lieb's lemma for nets]\label{OVBL-NV}
Let $X$, $E$, $E_0$, $F$, $T:E_0\to F$, and $J:X\to E$ satisfy the above hypothesis. 
Let $(g_\alpha)_{\alpha\in A}$ be a net in $X$ satisfying $g_\alpha\CX 0$, let $f\in X$ be such that 
$|Jf|,\Phi_{\varepsilon}g_\alpha,\Psi_{\varepsilon}f\in E_0$ for all 
$\varepsilon>0,\alpha\in A$, and let some $u\in F_+$ exist with $T\Phi_{\varepsilon}g_\alpha\le u$
for all $\varepsilon>0$, $\alpha\in A$. Then
$$
  T\bigg(|J(f+g_\alpha)-(Jf+Jg_\alpha)|\bigg)\OF 0 \ \ \ \ \ (\alpha\to\infty).
$$
\end{thm}

\begin{proof}
It follows from (\ref{3.1}) that
$$
  |J(f+g_\alpha)-(Jf+Jg_\alpha)|\le|J(f+g_\alpha)-Jg_\alpha|+|Jf|\le
  \varepsilon\Phi_{\varepsilon}g_\alpha+\Psi_{\varepsilon}f+|Jf|,
$$
and hence
$$
  |J(f+g_\alpha)-(Jf+Jg_\alpha)|-\varepsilon\Phi_{\varepsilon}g_\alpha\le\Psi_{\varepsilon}f+|Jf|\ \ \ \ \ (\varepsilon>0, \alpha\in A).
$$
Thus 
\begin{equation}\label{3.2}
  0\le w_{\varepsilon,\alpha}:=\bigg(|J(f+g_\alpha)-(Jf+Jg_\alpha)|-\varepsilon\Phi_{\varepsilon}g_\alpha\bigg)_+\le
	\Psi_{\varepsilon}f+|Jf|
\end{equation}
for all $\varepsilon>0$ and $\alpha\in A$.
It follows from (\ref{3.2}) and from $c_Xc_E-$continuity of $J$ and $\Phi_{\varepsilon}$, that $E_0\ni w_{\varepsilon,\alpha}\CE 0$ as $\alpha\to\infty$.
Furthermore, (\ref{3.2}) implies
\begin{equation}\label{3.3}
  |J(f+g_\alpha)-(Jf+Jg_\alpha)|\le w_{\varepsilon,\alpha}+\varepsilon\Phi_{\varepsilon}g_\alpha\ \ \ \ \ (\varepsilon>0, \alpha\in A).
\end{equation}
Since $T\ge 0$ and $T\Phi_{\varepsilon}g_\alpha\le u$, we get from (\ref{3.3})
\begin{equation}\label{3.4}
  0\le T\bigg(|J(f+g_\alpha)-(Jf+Jg_\alpha)|\bigg)\le Tw_{\varepsilon,\alpha}+{\varepsilon}T\Phi_{\varepsilon}g_\alpha\le
	Tw_{\varepsilon,\alpha}+{\varepsilon}u
\end{equation}
for all $\varepsilon>0$ and $\alpha\in A$.
Since $F$ is Dedekind complete and $T$ is $c_{E_0}o_F-$conti\-nuous, $Tw_{\varepsilon,\alpha}\OF 0$, and in view of (\ref{3.4})
$$
  0\le (o_F)-\limsup\limits_{\alpha\to\infty}\ T\bigg(|J(f+g_\alpha)-(Jf+Jg_\alpha)|\bigg)\le{\varepsilon}u\ \ \ \ \ (\forall \varepsilon>0).
$$
Then $T\bigg(|J(f+g_\alpha)-(Jf+Jg_\alpha)|\bigg)\OF 0$.
\end{proof}

\begin{enumerate}
  \item[$(1)$] Replacing nets by sequences one can obtain a sequential version of Theorem \ref{OVBL-NV}, 
	whose details are left to the reader. 
  \item[$(2)$] In the case of $F=\mathbb{R}$ and $X=E=L^0(\mu)$ with the almost everywhere convergence, $E_0=L^1(\mu)$,
  $Tf=\int f d\mu$, and $J:X\to E$ given by $Jf=j\circ f$, where $j:\mathbb{C}\to\mathbb{C}$ is continuous with $j(0)=0$ 
  such that for every $\varepsilon>0$ there exist two continuous functions 
  $\phi_{\varepsilon},\psi_{\varepsilon}:\mathbb{C}\to\mathbb{R}_+$ satisfying
$$
  |j(x+y)-j(x)|\le\varepsilon\phi_{\varepsilon}(x)+\psi_{\varepsilon}(y) \ \ \ \ \ (\forall x,y\in \mathbb{C}),
$$
  we obtain Theorem \ref{Thm1}, which is the classical Brezis - Lieb's lemma \cite[Thm.2]{BL1},
	from Theorem \ref{OVBL-NV}, by letting $\Phi_{\varepsilon}(f):=\phi_{\varepsilon}\circ f$ 
	and $\Psi_{\varepsilon}(f):=\psi_{\varepsilon}\circ f$.
\end{enumerate}

\end{document}